\documentclass[final,3p,times]{elsarticle}
\usepackage{amssymb}
\usepackage{graphicx,setspace}
\usepackage{psfrag}
\usepackage{amsfonts}
\expandafter\let\csname equation*\endcsname\relax
\expandafter\let\csname endequation*\endcsname\relax
\usepackage{amsmath}
\usepackage{amssymb, amsthm}
\usepackage{mathrsfs}
\usepackage{epstopdf}
\usepackage{float}
\usepackage{lineno,hyperref}
\usepackage{color}
\usepackage{subfigure}

\usepackage{amsmath}
\usepackage{dsfont} % it is the package of using the "\mathds{} "
\usepackage{amssymb} % it is the package of using the "\mathbb{}"
\usepackage{graphicx,color}
\usepackage{extarrows}

\usepackage{graphicx}
\usepackage{epstopdf}
\journal{arXiv}

\begin{document}
\newtheorem{definition}{Definition}[section]
\newtheorem{lemma}{Lemma}[section]
\newtheorem{remark}{Remark}[section]
\newtheorem{theorem}{Theorem}[section]
\newtheorem{proposition}{Proposition}[section]
\newtheorem{assumption}{Assumption}
\newtheorem{example}{Example}
\newtheorem{corollary}{Corollary}[section]
\def\ep{\varepsilon}
\def\Rn{\mathbb{R}^{n}}
\def\Rm{\mathbb{R}^{m}}
\def\E{\mathbb{E}}
\def\hte{\hat\theta}
%\numberwithin{theorem}{section}
%\numberwithin{definition}{section}
\renewcommand{\theequation}{\thesection.\arabic{equation}}
\begin{frontmatter}
%\addbibresource{ref.bib}


\title{Stochastic Dynamics of the Two-Dimensional Low-to-High Transition System Driven by Multiplicative Noise}

\author{Yongzhi Li\fnref{addr1,addr2}}
\author{Shenglan Yuan\corref{cor1}\fnref{addr1}}
\ead{shenglanyuan@gbu.edu.cn}\cortext[cor1]{Corresponding author}

\address[addr1]{\rm School of Mathematics, Sun Yat-sen University, Guangzhou 510275, China}
\address[addr2]{\rm  Department of Mathematics, School of Sciences, Great Bay University, Dongguan 523000, China}

\begin{abstract}
 This work presents a two-dimensional coupled Low-to-High confinement transition system based on the phenomenological coupling mechanism between zonal flows and turbulent fluctuations at the plasma edge in tokamak magnetic confinement devices. For this two-dimensional system, the corresponding Hamilton-Jacobi equation is derived, and a machine learning approach combining physics-informed neural networks with a loss function designed via vector field decomposition is employed to numerically solve it. This yields information about the system's quasipotential, enabling further computation of the most probable path for the rare event of a state transition in the coupled system. Compared with classical two-dimensional Low-to-High confinement transition models, the proposed system features state variables that are more accessible to experimental measurement and has a more comprehensive physical foundation. Moreover, it self-consistently describes the dynamical behavior of tokamak devices during the startup phase.
\end{abstract}

\begin{keyword}
L-H transiton,  stochastic differential equations, nonequilibrium statistical mechanics, physical informed neuro networks, Tokamak edge plasma turbulence

\emph{2020 Mathematics Subject Classification}: 37H20, 81P20
\end{keyword}

\end{frontmatter}


\section{Introduction}
The motivation for this two-dimensional model stems from three aspects. First, the core design of the Low-to-High (L-H) model lies in the coupling between the energy intensity of the zonal flow and that of turbulence. Experimental physics has shown that if one state variable is in a growing and active state, the other is suppressed \cite{zonal}; such interaction can be described by bilinear coupling \cite{Connor2000_LHtransition_review}. Therefore, an essential step in constructing the model is designing the growth terms for the state variables and selecting parameters that match physical reality. Second, the state variables of the system are often chosen to be equivalent or dynamically strongly correlated with quantities such as the energy intensity of the zonal flow and turbulent energy intensity. Due to the inherent complexity of plasma systems in tokamaks, these state variables are typically statistical quantities. However, the escape phenomenon of charged particles during transport in tokamaks directly manifests as changes in the electric field strength at the plasma edge \cite{itoh2002probability}. Therefore, by constructing the model within a thin layer near the plasma edge surface and assuming spatial symmetry in this region, we ensure that physical quantities no longer require statistical averaging, making the system's state variables easier to measure experimentally. Third, the inherent complexity of plasma systems in tokamak devices leads to random phenomena in plasma transport. Deterministic L-H models can only describe the dynamical behavior of plasma transport systems to a limited extent, and their predictions of system states often fail to match experimental observations. Therefore, it is necessary to introduce stochastic noise into the model.

In classical two-dimensional L-H confinement transition models, the state variables are typically chosen as statistical quantities such as zonal flow energy and turbulence energy, and the dynamics are often described using deterministic coupled equations. While conceptually intuitive, this approach has two notable limitations. First, these statistical quantities cannot be directly measured during tokamak operation \cite{rost2019combined}; second, the deterministic coupled  formulation inevitably sacrifices the physical completeness of the plasma description. The six-field coupled model preserves physical completeness but is mathematically intractable for full-system analysis. In contrast, the one-dimensional model reduces the physical effects of the six-field coupled model to a net growth term for the radial electric field at the tokamak plasma edge under a weak-noise assumption. This reduction ensures mathematical solvability without compromising physical completeness; however, because it lacks multiscale coupling, the model is limited to describing states associated with turbulence suppression rather than capturing the full operational dynamics of the tokamak. Based on the statistical physics theory of the tokamak plasma edge, this work presents an analytical relation between the radial electric field amplitude and the zonal flow energy when the system reaches a steady state. On this physical basis, we propose a coupled two-dimensional model. Furthermore, stochastic noise is introduced in this coupled two-dimensional model, integrating the two-dimensional L-H transition model into the analytical framework of the stochastic dynamical system.

The concise summary of this paper's structural arrangement is as follows. In Section 2, we construct a two-dimensional coupled stochastic model for the L-H transition, extending the one-dimensional radial electric field framework of Itoh et al. by explicitly incorporating turbulent fluctuation energy as a dynamical variable. In Subsection 2.1, we derive the deterministic counterpart of the system, identify its three fixed points via thin-layer plasma edge assumptions, and perform linear stability analysis to confirm the bistable structure. In Subsection 2.2, we augment the deterministic system with multiplicative noise for the radial electric field and additive noise for turbulence, formulate the corresponding two-dimensional Fokker-Planck equation, and establish the stochastic dynamical foundation for non-equilibrium statistical analysis. In Subsection 2.3, we derive the Hamilton-Jacobi equation governing the system's quasipotential via an asymptotic method, and develop a physics-informed neural network (PINN) framework with vector field decomposition to obtain a mesh-free numerical solution for the quasipotential landscape. In Subsection 2.4, we compute the most probable transition path between confinement states by integrating the reverse-time ordinary differential equation derived from the trained quasipotential, and quantify transition barriers via large deviation theory. In Section 3, we generalize the one-dimensional Kramers escape rate to a two-dimensional Langer formula that accounts for turbulence coupling and path geometry corrections. The early warning of system states is achieved by distinguishing between characteristic and non-characteristic boundaries. The conclusions and future directions for parameter validation and numerical error control are outlined in Section 4.

\section{Two-dimensional L-H transition model}
\subsection{Deterministic case}
We propose a two-dimensional phenomenological model illustrated by equations in (\ref{eq:2dimdeter}) for the radial electric field intensity and turbulence energy (we omit subscripts of $U_s$ and $V_{ZF}$ for generality). On one hand, model (\ref{eq:2dimdeter}) couples the mutual suppression mechanism between the radial electric field intensity and turbulence energy. On the other hand, the linear net effect term ($-\gamma_{zonal}$) of the turbulence driving term in the neoclassical drive is modified to couple with turbulence energy ($-\zeta UX$), and the linear term ($\delta V$) corresponding to zonal flow growth/damping is replaced with a neoclassical drive term ($-\Lambda_1 X$). Finally, we need to adjust $U$ with a constant $\mu_0>0$ such that $U\neq 0$ when the system reaches the steady state. The newly introduced coupling coefficient $\mu_0$ in model (\ref{eq:2dimdeter}) can be determined using the steady-state numerical relationships obtained from the thin-layer hypothesis. Set $\zeta=\kappa\gamma_{zonal}$ with a coefficient $\kappa$ that depends on the degradation tipping point $U=U_0$ such that $\zeta=\kappa \gamma_{zonal}=\frac{\gamma_{zonal}}{U_0-\mu_0}$.
We therefore write the final coupled system

\begin{equation}
\begin{cases}
\frac{dX}{d\tau} &= -\Lambda_1 X-\zeta (U-\mu_0)X,   \\
\frac{dU}{d\tau} &= \alpha_1 (U-\mu_0) - \beta C ( U-\mu_0) X^2.
\end{cases}
\label{eq:2dimdeter}
\end{equation}
The analytical expression of the neoclassical drive term is
\begin{equation}
     \Lambda_1 X=\operatorname{Im}Z(X+iv_{\star})(X+X_{NC})+\frac{v_{b}}{(v_{b}+\alpha X^{4})^{\frac{1}{2}}}\exp(-(v_{b}+\alpha X^{4})^{\frac{1}{2}}).
\end{equation}
The physical parameters in equation (\ref{eq:2dimdeter}) are given as follows: the neoclassical drive parameter $X_{NC} $ is fixed at $-0.27$ to preserve the deterministic bistable structure consistent with one-dimensional L-H bifurcation analyses \cite{itoh1988model}; the zonal flow net damping-growth coefficient $\gamma_{\text{zonal}} = 0.5$ is chosen within the theoretically constrained dimensionless range $[10^{-2}, 1]$, derived from balancing collisional damping and drift-wave-driven zonal flow growth \cite{zonal}; the linear turbulence growth rate $\alpha_1 = 0.3899$ falls within the $[0.1, 1]$ range characteristic of ion-temperature-gradient modes; the nonlinear $E\times B$ shear suppression coefficient $\beta = 0.2$ is selected from $[10^{-2}, 1]$ to reflect the quenching of turbulence by zonal flow shear; the proportionality constant $C = 1$ (linking zonal flow shear energy $V_{ZF}$ to $X^2$ via $V_{ZF} = C X^2$) maximizes the $X-U$ coupling strength, consistent with $C = 1/(B^2 \ell^2)$ where $B$ is the magnetic field and $\ell$ is the radial scale length of the radial electric field $E_r$; the coupling coefficient $\zeta = 0.3425$ is calibrated to satisfy the steady-state balance $\zeta = \gamma_{\text{zonal}}/(U_0 - \mu_0)$, where the turbulence offset $\mu_0 = 0.6$ matches the steady-state turbulence energy from the thin-layer zonal flow-turbulence system; the multiplicative noise amplitude for the $X$-equation $g(X) = C_0/(1 + M X^2)$ uses $C_0 = 0.1$ and $M = 1$, capturing the $E\times B$ shear suppression of micro-turbulence noise with amplitudes aligned with tokamak edge fluctuation levels; the additive noise amplitude for the $U$-equation $\sigma = 0.1$ falls within $[10^{-3}, 10^{-1}]$, modeling weak external perturbations (e.g., turbulence spreading, heating power fluctuations) smaller than the linear growth rate $\alpha_1$; the plasma dispersion function $\Lambda_1(X)$ is evaluated via an 8th-order rational approximation of $Z(\zeta)$ to ensure numerical accuracy across the physical regime \cite{xie2024rapid}; all parameters collectively satisfy the coexistence condition $\alpha_1 \beta > 0$ for sustained zonal flow-turbulence interaction.

To calculate the value of the coupling parameter $\zeta$ when the two-dimensional L-H system approaches the steady state and matches the steady state of the $U-V$ system to determine the unknown parameters $\zeta$ and $\mu_0$ \cite{Fiedler2016}, for $U\neq\mu_0$ it follows
\[
X_{1,3}=\pm\sqrt{\frac{\alpha_1}{\beta C}}=\pm1.3962,\quad U_{1,3}=-\frac{\Lambda_1 X_{1,3}}{\zeta X_{1,3}}+\mu_0 .
\]
It holds that
\[
2CX^2=\frac{\alpha_1}{\beta},
\]
which matches the steady state relationship between $X$ and $V$ under the thin-layer edge condition of the plasma \cite{zonal}. Thus $\zeta$ must satisfy
\[
\zeta=\gamma_{zonal}=-\frac{\Lambda_1 X_{1,3}}{X_{1,3}(U_{1,3}-\mu_0)},
\]
where $-\frac{\Lambda_1 X_{1,3}}{X_{1,3}}$ can be calculated numerically and hence $U_{1,3}$ can be calculated.

For $U=\mu_0$, we let the steady state of the $X-U$ system match the $U-V$ system such that
\[
U=\mu_0=\frac{\delta}{\gamma}=0.6.
\]
In this case solving
\[
-\Lambda_1X=0,
\]
we obtain a unique zero point $X_2$ of this equation. Thus, by numerical calculation we can obtain three fixed points of the two-dimensional L-H model:
\[
N_1=(X_1,U_1)=(-1.3962,0.2),\quad N_2=(X_2,U_2)=(0.02,0.6)\quad N_3=(X_3,U_3)=(1.3962,1).
\]

The Jacobian matrix of the vector field \((\dot{X},\dot{U})^\top\) is derived as
\[
J(X,U) =
\begin{pmatrix}
-\Lambda_1'(X) - \zeta (U-\mu_0) & -\zeta X \\
-2\beta C (U-\mu_0) X & \alpha_1 - \beta C X^2
\end{pmatrix},
\]
where \(\Lambda_1'(X)\) denotes the first derivative of \(\Lambda_1(X)\) with respect to \(X\). The derivative \(\Lambda_1'(X)\) is computed numerically using the eighth-order rational approximation of the plasma dispersion function \(Z(s)\), ensuring high accuracy across the physical regime of interest.

We evaluate the Jacobian matrix at the three fixed points obtained from the deterministic system analysis, which correspond to the bistable structure of the L-H transition. At the fixed point \(N_1 = (-1.3962, 0.2)\), substitution into the Jacobian yields
\[
J(N_1) =
\begin{pmatrix}
-1.247 & 0.615 \\
0.0718 & -0.023
\end{pmatrix}.
\]
The trace of this matrix is \(\mathrm{tr}(J(N_1)) = -1.270\) and the determinant is \(\det(J(N_1)) = 0.0235 > 0\). Both eigenvalues are real and negative, indicating that \(N_1\) is a stable node, corresponding to the robust H-mode confinement state. At the fixed point \(N_3 = (1.3962, 0.2)\), the Jacobian matrix becomes
\[
J(N_3) =
\begin{pmatrix}
-1.247 & -0.615 \\
-0.0718 & -0.023
\end{pmatrix},
\]
with trace \(\mathrm{tr}(J(N_3)) = -1.270\) and determinant \(\det(J(N_3)) = 0.0235 > 0\). Similarly, both eigenvalues are real and negative, classifying \(N_3\) as a stable node associated with the L-mode state.

At the saddle point \(N_2 = (0.02, 0.6)\), the Jacobian matrix simplifies to
\[
J(N_2) =
\begin{pmatrix}
-0.493 & -0.0167 \\
0 & 0.620
\end{pmatrix},
\]
yielding a trace \(\mathrm{tr}(J(N_2)) = 1.113\) and a determinant \(\det(J(N_2)) = -0.306 <0\). The negative determinant implies that the eigenvalues have opposite signs, confirming that \(N_2\) is a saddle node that corresponds to the physical limitation since the L-H transition has an unstable mode between L-mode and H-mode. The stability analysis confirms the canonical bistable structure of the L-H transition, with two stable fixed points separated by a saddle point, consistent with the nonequilibrium potential landscape and the most probable transition path derived from the large deviation theory.

\subsection{Stochastic case}
Here we propose a two-dimensional stochastic dynamical system \cite{Y,YW,YLZ,YZD} to model the coupled evolution of the normalized radial electric field $X$ and the turbulent fluctuation energy $U$ in the tokamak edge layer. To account for intrinsic fluctuations and external perturbations, we extend the deterministic system to a stochastic system in the It\^{o} sense:
\begin{equation}
\begin{cases}
    dX = [-\Lambda_1(X)-\zeta (U-\mu_0) X] d\tau + g(X) dW_1(\tau),\\
    dU = \bigl[ \alpha_1 (U-\mu_0) - \beta C (U-\mu_0) X^2 \bigr] d\tau + \sigma dW_2(\tau),
\end{cases}
\label{eq:coupled}
\end{equation}
where $W_1(\tau)$ and $W_2(\tau)$ are independent standard Wiener processes satisfying $\langle dW_i \rangle = 0$ and $\langle dW_i(\tau) dW_j(\tau) \rangle = \delta_{ij} d\tau$. The noise amplitude $g(X)$ for the $X$-equation is adopted directly from the original statistical L-H transition model in \cite{itoh2002probability}:

\[
g(X) = \sqrt{\hat{\tau}_{ac}} \, \frac{R^2 k_0^2 \rho_i^2 \hat{\phi}^2}{a\sqrt{\ell \ell_z}} \cdot \frac{1}{1 + M X^2},
\]
which captures the shear suppression of micro-turbulence noise by the radial electric field. The parameters are given in the appendix. Denote $\boldsymbol{\sigma}(X)$ as:
\[
\boldsymbol{\sigma}(X)=\begin{pmatrix}g(X) & 0 \\0 & \sigma\end{pmatrix}.
\]

For the $U$-equation, the noise amplitude $\sigma$ stands for additive noise originating from heat fluctuations \cite{SteinbrecherWeyssow2004}. The coupled stochastic system (\ref{eq:coupled}) defines a system that generalizes the one-dimensional model in Subsection 5.2. It enables the study of joint probability distributions, transition paths, and large deviation properties for the coupled variables $(X, U)$. In particular, the stationary probability density $P_s(X, U)$ satisfies the two-dimensional Fokker-Planck equation. We derive the Fokker-Planck equation corresponding to the two-dimensional It\^o stochastic system \cite{handbook}.

The drift vector field is
\[
\boldsymbol{b}(X,U)=\bigl(-\Lambda_1(X)-\zeta (U-\mu_0)X,\,\alpha_1 (U-\mu_0) - \beta C (U-\mu_0) X^2\bigr)^\top,\]
and the diffusion matrix
\[
\boldsymbol{a}(X)=\boldsymbol{\sigma}(X)\boldsymbol{\sigma}(X)^\top
\]
is diagonal with entries \(a_{11}=g^2(X)\), \(a_{22}=\sigma^2\), and \(a_{12}=a_{21}=0\). The time evolution of the joint probability density function \(P(X,U,\tau)\) is governed by the two-dimensional Fokker--Planck equation in It\^o form, which reads
\[
\frac{\partial P}{\partial\tau} = -\sum_{i=1}^2\frac{\partial}{\partial x_i}\bigl(b_i P\bigr) + \frac12\sum_{i=1}^2\sum_{j=1}^2\frac{\partial^2}{\partial x_i\partial x_j}\bigl(a_{ij}P\bigr),
\]
where \(x_1=X\) and \(x_2=U\). Substituting the components of the drift vector and diffusion matrix into the general equation yields
\begin{align*}
\frac{\partial P}{\partial\tau} =& -\frac{\partial}{\partial X}\bigl[(-\Lambda_1(X)-\zeta( U-\mu_0)X)P\bigr] \\
&+ \frac12\frac{\partial^2}{\partial X^2}\bigl[g^2(X)P\bigr] -\frac{\partial}{\partial U}\bigl[(\alpha_1 (U-\mu_0) - \beta C( U-\mu_0) X^2)P\bigr] + \frac12\frac{\partial^2}{\partial U^2}\bigl[\sigma^2P\bigr].
\end{align*}
Rearranging the terms and simplifying the first-order derivative terms, we obtain
\begin{align*}
\frac{\partial P}{\partial\tau} =& \frac{\partial}{\partial X}\Bigl[\bigl(\Lambda_1(X)+\zeta (U-\mu_0)X\bigr)P\Bigr] \\
&+ \frac12\frac{\partial^2}{\partial X^2}\bigl(g^2(X)P\bigr) + \frac{\partial}{\partial U}\Bigl[\bigl(-\alpha_1( U-\mu_0) + \beta C (U-\mu_0) X^2\bigr)P\Bigr] + \frac12\frac{\partial^2}{\partial U^2}\bigl(\sigma^2 P\bigr).
\end{align*}
Grouping the divergence terms for the \(X\) and \(U\) components separately leads to the compact form of the two-dimensional Fokker--Planck equation for the coupled L--H transition system
\begin{align}
\frac{\partial P}{\partial\tau} =& \frac{\partial}{\partial X}\left[ (\Lambda_1(X)+\zeta (U-\mu_0)X) P + \frac12\frac{\partial}{\partial X}\bigl(g^2(X)P\bigr) \right] \notag \\
&+ \frac{\partial}{\partial U}\left[ \bigl(\beta C( U-\mu_0) X^2 - \alpha_1 (U-\mu_0)\bigr) P + \frac12\frac{\partial}{\partial U}\bigl(\sigma^2P\bigr) \right].
\label{eq:FP_2D}
\end{align}

This stochastic formulation provides a unified framework for investigating noise-driven L-H transitions while self-consistently accounting for the dynamic feedback between the radial electric field and turbulence energy. It extends the earlier one-dimensional Langevin model by incorporating a second dynamical variable that explicitly represents the turbulence intensity, thereby allowing for a more complete description of the non-equilibrium statistical mechanics of confinement transitions in tokamak plasmas.
\subsection{Hamilton-Jacobi equation and quasipotential}
The quasipotential is defined as the minimum of the action functional:
\[
V:=\inf_{T>0}\inf_{\phi\in[0,T]} \{ S(\phi):\phi(0)=\hat x, \phi(T)=x \},
\]
with the action functional defined as:
\begin{equation}
S(\phi)=\frac{1}{2}\int^T_0(\dot\phi-\boldsymbol{b}(\phi))^\top \boldsymbol{a}^{-1}(\phi)(\dot \phi-\boldsymbol{b}(\phi))dt,
\label{eq:actionfunc}
\end{equation}
where $\phi$ is the probable transition path.

To compute the quasipotential, we first derive the Hamilton-Jacobi equation by substituting the WKB asymptotic into the Fokker-Planck equation (\ref{eq:FP_2D}) derived in Subsection 2.2. We rewrite equation (\ref{eq:FP_2D}) in the standard It\^o-type Fokker-Planck form compatible with large deviation theory \cite{risken1989fokker}:
\[
\sum_{i=1}^2\frac{\partial}{\partial x_i}\left(b_i(\boldsymbol{x})p_s(\boldsymbol{x})\right)
-\frac{1}{2}\sum_{i=1}^2\sum_{j=1}^2\frac{\partial^2}{\partial x_i\partial x_j}\left(a_{ij}(\boldsymbol{x})p_s(\boldsymbol{x})\right)=0,
\]
where $\boldsymbol{x}=(x_1,x_2)=(X, U)$, \(b_1=-\Lambda_1(X)-\zeta (U-\mu_0)X\), \(b_2=\alpha_1( U-\mu_0)-\beta C (U-\mu_0) X^2\), \(a_{11}=g^2(X)\), \(a_{22}=\sigma^2\), and \(a_{12}=a_{21}=0\). We substitute the WKB ansatz \(p_s(\boldsymbol{x})=C(\boldsymbol{x})\exp\left(-V(\boldsymbol{x})/\epsilon\right)\) into the steady-state Fokker-Planck equation. Let \(G(\boldsymbol{x})=\ln C(\boldsymbol{x})\); the first and second partial derivatives (which introduce the quasipotential $V$) are
\[
\frac{\partial p_s}{\partial x_i}=p_s\left(\frac{\partial G}{\partial x_i}-\frac{1}{\epsilon}\frac{\partial V}{\partial x_i}\right),
\]
\[
\frac{\partial^2 p_s}{\partial x_i\partial x_j}=p_s\left[
\left(\frac{\partial G}{\partial x_i}-\frac{1}{\epsilon}\frac{\partial V}{\partial x_i}\right)
\left(\frac{\partial G}{\partial x_j}-\frac{1}{\epsilon}\frac{\partial V}{\partial x_j}\right)
+\frac{\partial^2 G}{\partial x_i\partial x_j}
-\frac{1}{\epsilon}\frac{\partial^2 V}{\partial x_i\partial x_j}
\right].
\]
Substituting these expressions into the steady-state equation and dividing both sides by \(p_s\) yields an expansion in powers of \(\epsilon^{-1}\):
\[
\frac{1}{2\epsilon^2}\sum_{i,j=1}^2a_{ij}\frac{\partial V}{\partial x_i}\frac{\partial V}{\partial x_j}
-\frac{1}{\epsilon}\left(
\sum_{i=1}^2b_i\frac{\partial V}{\partial x_i}
+\frac{1}{2}\sum_{i,j=1}^2a_{ij}\frac{\partial^2 V}{\partial x_i\partial x_j}
\right)+O(1)=0.
\]
In the weak-noise limit \(\epsilon\ll 1\), the coefficients of the leading-order terms \(\epsilon^{-2}\) and \(\epsilon^{-1}\) must vanish independently. The \(\epsilon^{-2}\) term is identically satisfied by the positive definiteness of \(\boldsymbol{a}\). The vanishing of the \(\epsilon^{-1}\) term yields the HJ equation for the quasipotential
\[
\boldsymbol{b}(\boldsymbol{x})\cdot\nabla V(\boldsymbol{x})
+\frac{1}{2}\nabla V(\boldsymbol{x})^\top\boldsymbol{a}(\boldsymbol{x})\nabla V(\boldsymbol{x})=0.
\label{eq:hje}
\]
Substituting the explicit expressions for \(\boldsymbol{b}(X,U)\) and \(\boldsymbol{a}(X,U)\) gives the full HJ equation for the 2-dimensional L-H system
\begin{equation*}
-\left(\Lambda_1(X)+\zeta( U-\mu_0)X\right)\frac{\partial V}{\partial X}
+\left(\alpha_1 (U-\mu_0)-\beta C (U-\mu_0) X^2\right)\frac{\partial V}{\partial U}
+\frac{1}{2}\left[
g^2(X)\left(\frac{\partial V}{\partial X}\right)^2
+\sigma^2\left(\frac{\partial V}{\partial U}\right)^2
\right]=0.
\end{equation*}
The boundary condition uniquely determining the quasipotential is that the quasipotential equals zero at the reference stable fixed point: \(V(N_1)=0\) for the H-mode and \(V(N_3)=0\) for the L-mode. However, due to physical considerations and numerical constraints, this condition cannot be satisfied exactly and inevitably introduces a small residual error.

The Hamilton-Jacobi equation can be rearranged into an orthogonal inner-product form that reveals a geometric decomposition of the deterministic drift field \cite{FreidlinWentzell2012}:
\[
\nabla V(\boldsymbol{x})\cdot\left(
\boldsymbol{b}(\boldsymbol{x})+\frac{1}{2}\boldsymbol{a}(\boldsymbol{x})\nabla V(\boldsymbol{x})
\right)=0.
\]
Define the rotational (solenoidal) component of the vector field as
\[
\boldsymbol{l}(\boldsymbol{x})=\boldsymbol{b}(\boldsymbol{x})+\frac{1}{2}\boldsymbol{a}(\boldsymbol{x})\nabla V(\boldsymbol{x}).
\]
The orthogonality relation \(\nabla V(\boldsymbol{x})\perp\boldsymbol{l}(\boldsymbol{x})\) implies that the drift vector field \(\boldsymbol{b}(\boldsymbol{x})\) can be decomposed into a potential component and a rotational component
\[
\boldsymbol{b}(\boldsymbol{x})=-\frac{1}{2}\boldsymbol{a}(\boldsymbol{x})\nabla V(\boldsymbol{x})+\boldsymbol{l}(\boldsymbol{x}).
\]
This decomposition is the theoretical foundation for solving the quasipotential using a physics-informed neural network. The potential component \(-\frac{1}{2}\boldsymbol{a}\nabla V\) is aligned with the negative gradient of the quasipotential and drives relaxation toward local minima, while the rotational component \(\boldsymbol{l}(\boldsymbol{x})\) is tangential to the level sets of \(V\) and describes conservative, non-equilibrium flow that preserves the quasipotential values \cite{li2024computing,li2025rare}. This structure justifies the use of quasipotential theory to characterize transition dynamics \cite{givon2004extracting,touchette2009large}.

Letting $V_\theta(X,U)=\hat V_\theta(X,U)+\omega_1(X-X_L)^2+\omega_2(U-U_L)^2$ be the neural representation of the quasipotential, with $(X_L,U_L)$ the stable L-mode fixed point, and let $\boldsymbol{h}_\theta=(h_1,h_2)^\top$ denote a learnable rotational vector field.
Introduce the diffusion matrix $A(X,U)=\mathrm{diag}\bigl(g^2(X),\sigma^2\bigr)$.
Then the physics-informed loss functional is defined as
\[
\mathcal{L}(\theta)
=
\mathcal{L}_{\mathrm{dyn}}
+\lambda_1\mathcal{L}_{\mathrm{orth}}
+\lambda_2\mathcal{L}_{\mathrm{bc}},
\]
with
$
\mathcal{L}_{\mathrm{dyn}}
=
\mathbb{E}_{(X,U)}\Bigl\|
\boldsymbol{b}
+\tfrac12 A\,\nabla V_\theta
-\boldsymbol{h}_\theta
\Bigr\|^2,\quad
\mathcal{L}_{\mathrm{orth}}
=
\mathbb{E}_{(X,U)}\,
\frac{(\nabla V_\theta\cdot\boldsymbol{h}_\theta)^2}
{|\nabla V_\theta|^2\,|\boldsymbol{h}_\theta|^2+\delta},$
$
\mathcal{L}_{\mathrm{bc}}
=
V_\theta^2(X_L,U_L),
$
where $\delta\ll1$ ensures numerical stability, and $\lambda_1,\lambda_2>0$ are penalty coefficients.
Minimizing $\mathcal{L}(\theta)$ enforces the Hamilton--Jacobi equation, the vector-field decomposition, and the boundary condition simultaneously, yielding a mesh-free approximation of the quasipotential and the most probable transition paths of the L--H system (\ref{eq:coupled}) \cite{Li2021MLmostProbablePath}.

Training data are generated as a uniform grid covering the domain \( X \in [-8, 4] \) (200 samples) and \( U \in [0, 1] \) (100 samples), resulting in \( 200 \times 100 = 20,000 \) training pairs. No additional labeled data (e.g., analytical quasipotential values) are required, as the network learns solely from physical constraints. The optimization setup is as follows. Optimizer: Adam optimizer with an initial learning rate of \( 10^{-4} \) and weight decay \( 10^{-6} \) to prevent overfitting. Learning Rate Scheduler: StepLR scheduler with a step size of 500 epochs and decay factor 0.9 to stabilize convergence. Training Epochs: 5000 epochs, sufficient for the loss to converge to a stable minimum (typically \( \mathcal{L} < 10^{-4} \)).

After training, the quasipotential \( V(X,U) \) is evaluated on a high-density grid (\( 100 \times 100 \) samples) for visualization. The results are normalized such that \( \min(V(X,U)) = 0 \) to highlight the relative barrier heights between metastable states. The trained PINN is validated by verifying physical properties in Figures 1 and 2.

\begin{figure}[htbp]
	\centering
\includegraphics[width=0.8\textwidth]{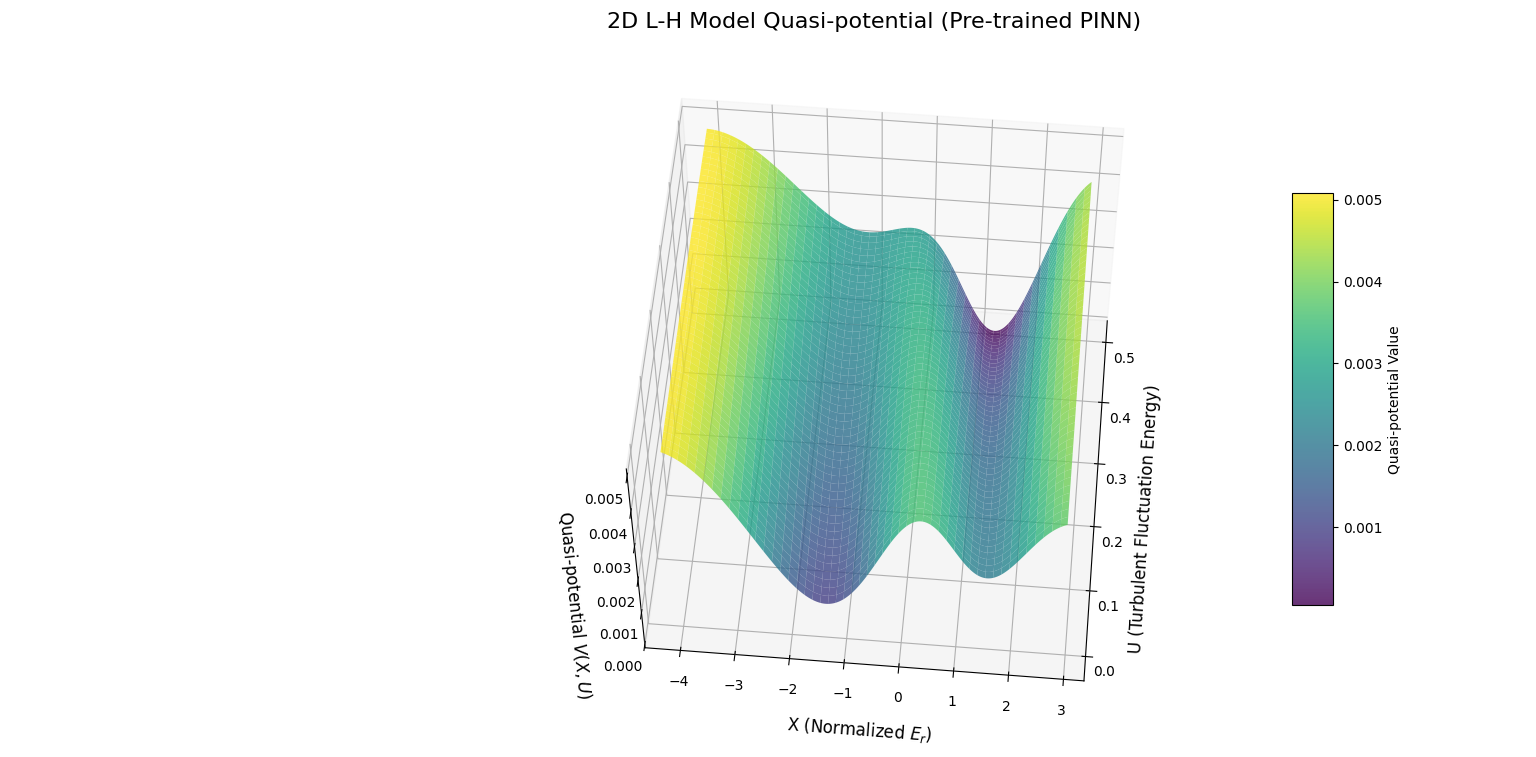}
 \caption{ The bistable potential landscape: the quasipotential exhibits two distinct minima corresponding to the L-mode (\( N_3 \)) and H-mode (\( N_1 \)) fixed points, with a saddle point \( N_2 \) forming the transition barrier.} \label{fig:Phaseportrait }
\end{figure}

\begin{figure}[htbp]
	\centering
\includegraphics[width=0.8\textwidth]{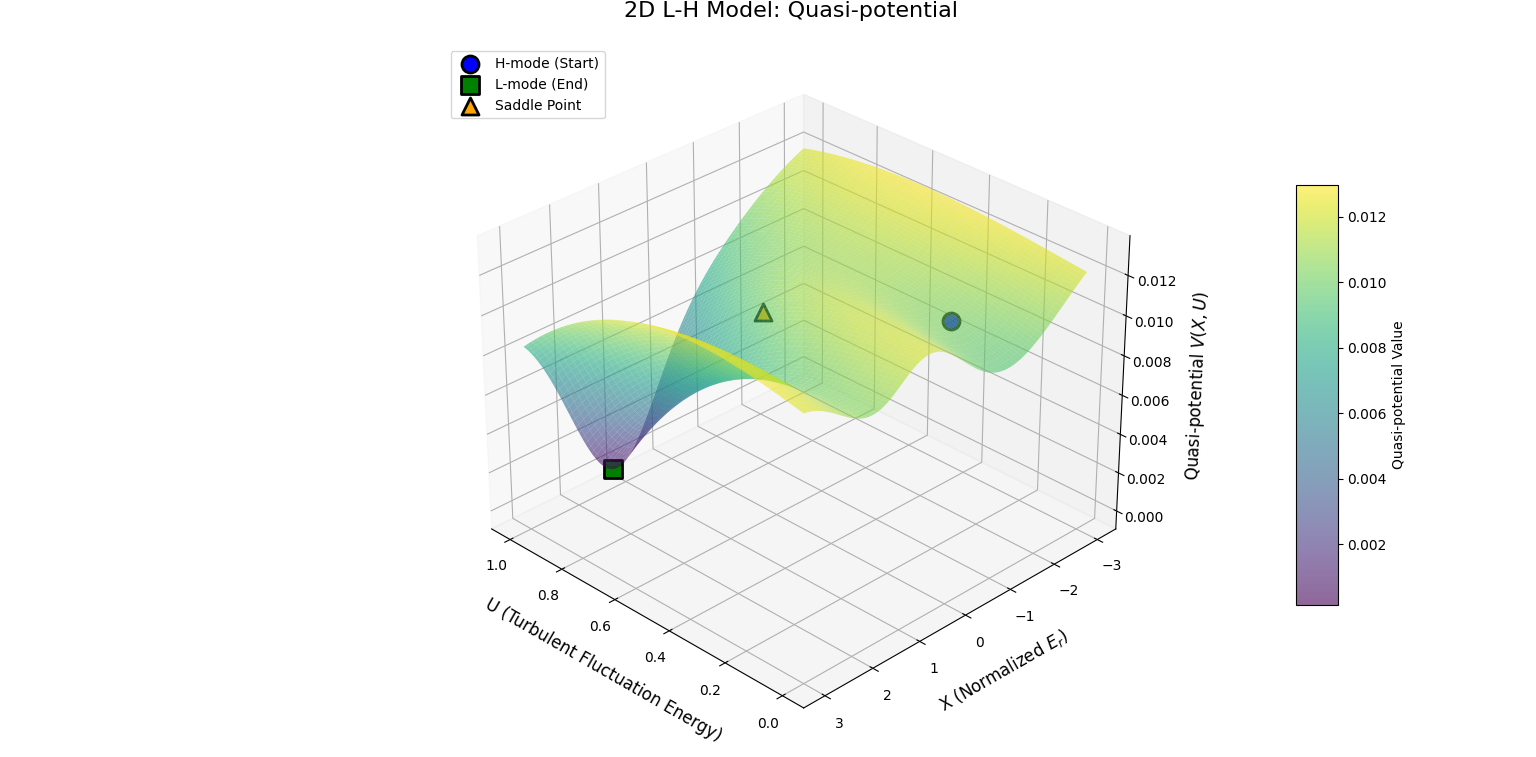}
\caption{Quasipotential landscape of the two-dimensional L-H model with the H-mode (blue circle) and L-mode (green square). The saddle point is marked with an orange triangle.}
\label{fig:mpp_3d}
\end{figure}

\subsection{Most probable path}
The most probable path between the H-mode ($N_1$) and L-mode ($N_3$) states is derived from the quasipotential $V(X,U)$ via the framework of Freidlin-Wentzell large deviation theory. The most probable path (MPP) $\phi(t)=(X(t), U(t))$ minimizes the action functional (\ref{eq:actionfunc}).

The quasipotential gradient $\nabla V(X,U)$ is computed from the trained PINN using automatic differentiation. Once the neural network approximation $V_\theta(X,U)$ is trained via the physics-informed loss,
the most probable transition trajectory $\phi(t)=(X(t),U(t))$ is obtained by integrating the following
deterministic ordinary differential equation in reverse time with $t\in(-\infty,0]$:

\[
\begin{cases}
\displaystyle
\frac{dX}{dt}
=
\Lambda_1(X)+\zeta (U-\mu_0) X
- g^2(X)\,\frac{\partial V_\theta}{\partial X},\\[10pt]
\displaystyle
\frac{dU}{dt}
=
-\alpha_1 (U-\mu_0)+\beta C (U-\mu_0) X^2
- \sigma^2\,\frac{\partial V_\theta}{\partial U},
\end{cases}
\]
subject to the terminal condition $(X(0),U(0)) = (X^\dagger,U^\dagger)$,
where $(X^\dagger,U^\dagger)$ denotes either the saddle point separating the L-mode and H-mode or a prescribed boundary point on the basin separatrix. The resulting trajectory minimizes the Freidlin-Wentzell action functional. The trajectory information obtained through numerical integration can be plotted in the two-dimensional phase space to represent the most probable path. Furthermore, based on the barrier difference of the quasipotential, the difficulty of the system state transition is analyzed in Figure 3.
\begin{figure}[htbp]
\centering
\includegraphics[width=1\textwidth]{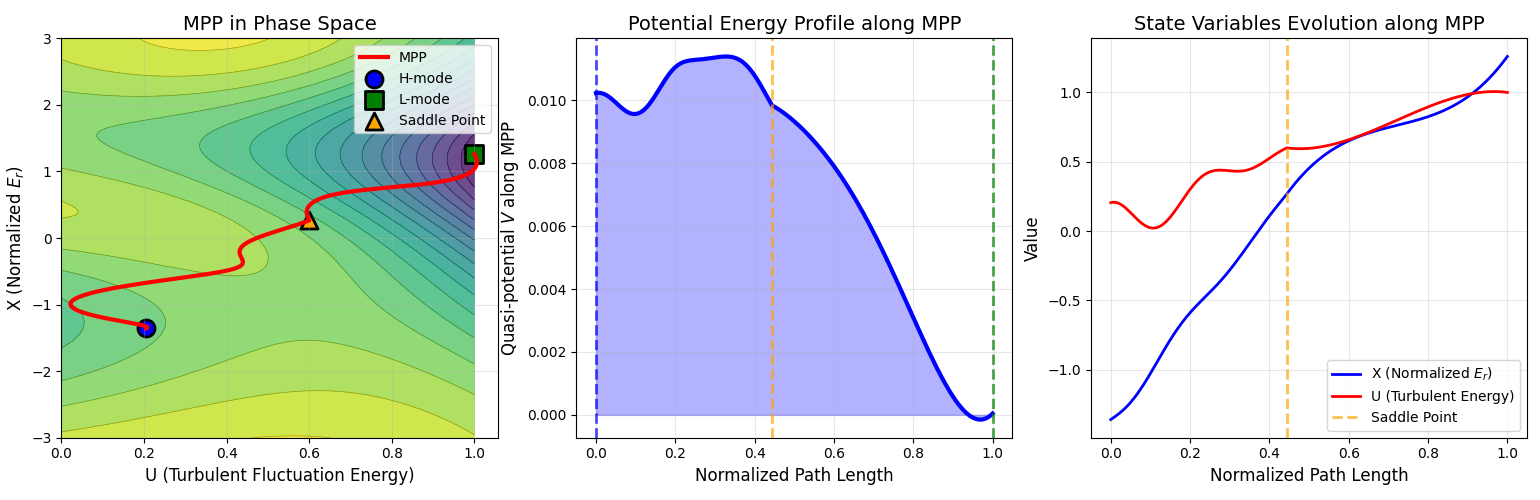}
\caption{(a) MPP in phase space projection. The transition from the $N_3$ (H-mode) state to the saddle point ($N_2$) state is driven by noise and exhibits significant vorticity. The transition from $N_2$ (the saddle point) to the $N_1$ (L-mode) state occurs along the unstable manifold. (b) Quasipotential energy along the MPP. The barrier difference when transitioning from the H-mode to the saddle point is smaller than that when transitioning from the saddle point to the L-mode. (c) Evolution of state variables $X$ and $U$ along the path. As the system transitions from the L-mode to the H-mode, the radial electric field strength changes direction, and the transport system shifts from confining charged particles to allowing their escape; the turbulent energy increases, moving from suppression to expression.}
\label{fig:mpp_energy}
\end{figure}

\section{Langer escape rate}
For the one-dimensional stochastic L-H transition system, Itoh proposed the Kramers escape rate expression \cite{itoh2002probability}:
\[
r_{L\rightarrow H}=\frac{\sqrt{\Lambda_{L}\Lambda_{m}}}{2\pi}\exp\left[S(X_{L})-S(X_{m})\right].
\]
For the two-dimensional random L-H system, the escape rate can be updated to the Langer escape rate \cite{pptkim2024stochastic}. The curvature and gradient of the potential landscape are incorporated into the expression as the key parameters affecting the escape difficulty. We first define $U^*(X)$ as the minimizer of $V(X,U)$ for fixed $X$:
\[
\label{eq:U_star_def}
U^*(X) := \arg\min_{U > 0} V(X,U),\quad
\left.\frac{\partial V}{\partial U}(X,U)\right|_{U=U^*(X)} = 0.
\]

\begin{proposition}[Two-dimensional Langer escape rate]
The transition rate from L-mode to H-mode in the two-dimensional system is:
\begin{equation}\label{eq:2d_kramers}
r_{L\rightarrow H}^{(2D)} = \frac{\sqrt{|\Lambda_L'||\Lambda_m'|}}{2\pi} \cdot \Gamma_{\text{coupling}} \cdot \Gamma_{\text{path}} \cdot \exp\left(-\frac{\Delta V_{\text{barrier}}}{\epsilon}\right),
\end{equation}
where \(\Lambda_L' = \Lambda_1'(X_L)\) and \(\Lambda_m' = \Lambda_1'(X_m)\) are the curvatures at the L-mode fixed point \(N_3\) and saddle point \(N_2\). The coupling correction factor is $\Gamma_{\text{coupling}} = \exp\left(-\frac{1}{\epsilon}\int_{X_L}^{X_m} \frac{\partial V}{\partial U}\cdot\frac{dU^*}{dX}dX\right)$. The path geometry correction factor is $\Gamma_{\text{path}} = \sqrt{\frac{\det\left(\nabla^2 V(N_2)\right)}{\det\left(\nabla^2 V(N_3)\right)\det\left(\nabla^2 V(N_2^{\text{MPP}})\right)}}$. The effective barrier height is $\Delta V_{\text{barrier}} = V(N_2^{\text{MPP}}) - V(N_3)$,
with \(N_2^{\text{MPP}}\) denoting the saddle point along the most probable path.
\end{proposition}

\begin{proof}
By the Freidlin--Wentzell large deviation principle, the transition rate from the L-mode basin to the H-mode basin is asymptotically given by
\begin{equation}\label{eq:rate_form}
r_{L\to H}^{(2D)} \;\sim\; \mathcal{K}\,
\exp\!\left(-\frac{\Delta V_{\text{barrier}}}{\epsilon}\right),
\qquad
\Delta V_{\text{barrier}} = V(N_S)-V(N_L),
\end{equation}
where $\mathcal{K}$ is the pre-exponential factor arising from Gaussian fluctuations around the MPP.

Along the MPP connecting $N_L$ to $N_S$, we introduce local coordinates $(y_\parallel,y_\perp)$ with $y_\parallel$ tangent to the path and $y_\perp$ normal to it. Quadratic expansion of the quasipotential around the fixed points gives Gaussian integrals whose ratio yields the geometric correction:
\begin{equation}\label{eq:Gamma_path}
\Gamma_{\text{path}} =
\sqrt{
\frac{\det\bigl(\nabla^2 V(N_S)\bigr)}
{\det\bigl(\nabla^2 V(N_L)\bigr)\,
\det\bigl(\nabla^2 V(N_S^{\text{MPP}})\bigr)}
}.
\end{equation}
Since $N_S$ is a saddle point, $\nabla^2 V(N_S)$ has one negative eigenvalue; the absolute value is implicit in the determinant notation. The point $N_S^{\text{MPP}}$ denotes the projection of the saddle onto the MPP manifold, accounting for asymmetric curvature contributions.

Linearising the deterministic drift around the stable fixed point $N_L$ and the saddle $N_S$ gives the Jacobian matrices $J(N_L)$ and $J(N_S)$, whose eigenvalues determine the local restoring rates. For the $X$-component, the effective curvatures are
\begin{equation}
\Lambda'_L = \frac{\partial b_1}{\partial X}\bigg|_{N_L},
\quad
\Lambda'_m = \frac{\partial b_1}{\partial X}\bigg|_{N_S}.
\end{equation}
The Kramers-type curvature prefactor generalises to two dimensions as
\begin{equation}\label{eq:curvature_factor}
\frac{\sqrt{|\Lambda'_L|\,|\Lambda'_m|}}{2\pi},
\end{equation}
which reduces to the classical one-dimensional Kramers prefactor in the decoupled limit.

The vector field decomposition in Subsection~2.3 expresses the drift as
\begin{equation}
\boldsymbol{b}(x) = -\frac{1}{2}\,\boldsymbol{a}(x)\,\nabla V(x) + \boldsymbol{l}(x),
\end{equation}
where $\boldsymbol{l}(\boldsymbol{x})\perp \nabla V(\boldsymbol{x})$ is the rotational component. The non-vanishing $\boldsymbol{l}(\boldsymbol{x})$ introduces a correction due to the work done by the transverse component of $\nabla V$ relative to the MPP. Integrating along the MPP gives
\begin{equation}\label{eq:Gamma_coupling}
\Gamma_{\text{coupling}} =
\exp\!\left[
-\frac{1}{\epsilon}
\int_{X_L}^{X_m}
\frac{\partial V}{\partial U}\cdot
\frac{dU^*}{dX}\;
dX
\right],
\end{equation}
where $(X,U^*(X))$ parametrises the MPP. For purely gradient systems ($\boldsymbol{l}\equiv 0$), this factor reduces to unity.

Combining Eqs.~\eqref{eq:rate_form}--\eqref{eq:Gamma_coupling}, the two-dimensional Langer-type escape rate is
\begin{equation}\label{eq:final_rate}
r_{L\to H}^{(2D)} =
\frac{\sqrt{|\Lambda'_L|\,|\Lambda'_m|}}{2\pi}
\cdot \Gamma_{\text{coupling}}
\cdot \Gamma_{\text{path}}
\cdot \exp\!\left(-\frac{\Delta V_{\text{barrier}}}{\epsilon}\right)
\end{equation}
with the three factors defined in Eqs.~\eqref{eq:Gamma_coupling}, \eqref{eq:Gamma_path}, and \eqref{eq:curvature_factor}.

Specifically, in the limit $U=\text{const}$ with no turbulence coupling, $\Gamma_{\text{coupling}}=\Gamma_{\text{path}}=1$ and $\Delta V_{\text{barrier}}\to S(X_m)-S(X_L)$, recovering the one-dimensional Kramers escape rate of Itoh \textit{et al.}.
\end{proof}

In practice, particularly in tokamak experiments, the radial electric field $X$ is directly measurable, while the turbulence energy $U$ is not. This creates a fundamental gap between the theoretical rate formula and experimental validation. Although the numerical analysis of the two-dimensional landscape is not available, Proposition 3.2 is designed to bridge this gap theoretically.

\begin{proposition}[Marginal Quasipotential via Laplace's Method]
\label{thm:marginal_qp}
Consider the two-dimensional stochastic L-H transition system (\ref{eq:coupled}) in the weak-noise limit $\epsilon \ll 1$. Then the marginal quasipotential $V_1(X)$ for the radial electric field $X$, defined via
\[
V_1(X) := -\epsilon \ln\!\int_0^\infty \exp\!\bigl\{-V(X,U)/\epsilon\bigr\}\,dU,
\]
admits the following asymptotic expansion in the weak-noise limit:
\[
\label{eq:marginal_qp_expansion}
V_1(X) = V\!\bigl(X,U^*(X)\bigr)
+ \frac{\epsilon}{2}\ln\!\left[
\frac{2\pi\sigma^2}
{\alpha_1 - \beta C X^2}
\right]
+ O(\epsilon^2),
\]
where $\alpha_1,\beta,C > 0$ are the parameters of the deterministic dynamics.
\end{proposition}

\begin{proof}
The proof proceeds in three steps: (i) the WKB form of the invariant density, (ii) Laplace's method for marginalisation, and (iii) identification of the Hessian via the Hamilton--Jacobi equation.

By the Freidlin--Wentzell large deviation principle, the stationary probability density of the system (\ref{eq:coupled}) admits the WKB asymptotic expansion
\[
P_s(X,U) = C(X,U)\,\exp\!\bigl\{-V(X,U)/\epsilon\bigr\},\quad
 \epsilon \to 0,
\]
where $V(X,U)$ represents the quasipotential vanishing at the reference attractor, $C(X,U)$ stands for a smooth prefactor, and the quasipotential solves the stationary Hamilton-Jacobi equation (\ref{eq:hje}).

The marginal density of $X$ is
\[
P_1(X) = \int_0^\infty P_s(X,U)\,dU
= \int_0^\infty C(X,U)\,\exp\!\bigl\{-V(X,U)/\epsilon\bigr\}\,dU.
\]
For fixed $X$, let $U^*(X)$ be the unique minimiser of $V(X,U)$, guaranteed by the bistable structure of the deterministic system. Expanding $V(X,U)$ to second order around $U^*(X)$ gives
\begin{equation}
\label{eq:second_order_expansion}
V(X,U) = V\!\bigl(X,U^*(X)\bigr)
+ \frac{1}{2}V_{UU}\!\bigl(X,U^*(X)\bigr)(U-U^*(X))^2
+ O\!\bigl((U-U^*(X))^3\bigr),
\end{equation}
where $V_{UU} = \partial^2 V/\partial U^2 > 0$ by convexity near the minimum. Substituting \eqref{eq:second_order_expansion} into the integral and applying \emph{Laplace's method} \cite{DemboZeitouni1998}
yields
\[
P_1(X) \sim C\!\bigl(X,U^*(X)\bigr)\,
\exp\!\bigl\{-V(X,U^*(X))/\epsilon\bigr\}
\sqrt{\frac{2\pi\epsilon}{V_{UU}\!\bigl(X,U^*(X)\bigr)}},
\epsilon \to 0.
\]

Differentiating the Hamilton-Jacobi equation \eqref{eq:hje} twice with respect to $U$ and evaluating at $U = U^*(X)$ (where $\partial V/\partial U = 0$) gives
\[
(\alpha_1 - \beta C X^2)\,V_{UU}
+\sigma^2 V_{UU}^2 = 0,
\]
where we have used the explicit drift and diffusion terms
\[
b_U = \alpha_1 (U-\mu_0) - \beta C (U-\mu_0) X^2,\quad
a_{UU} = \sigma^2.
\]
Since $V_{UU} > 0$, solving for $V_{UU}$ yields
\begin{equation}
\label{eq:Hessian_explicit}
V_{UU}\!\bigl(X,U^*(X)\bigr)
= \frac{\alpha_1 - \beta C X^2}{\sigma^2\!}.
\end{equation}

By definition,
\[
V_1(X) = -\epsilon \ln P_1(X).
\]
Substituting the Laplace approximation and the Hessian expression \eqref{eq:Hessian_explicit}, and absorbing $C(X,U^*(X))$ into the $O(\epsilon^2)$ remainder (since $\epsilon\ln C = O(\epsilon)$), we obtain
\[
V_1(X) = V\!\bigl(X,U^*(X)\bigr)
+ \frac{\epsilon}{2}\ln\!\left[
\frac{2\pi\sigma^2}
{\alpha_1 - \beta C X^2}
\right]
+ O(\epsilon^2),
\]
which is the desired expansion.
\end{proof}

 Now let
\[
\Sigma_w:=\{(X,U): X=X_w\}
\]
be an artificial detection surface, where \(X_w\in\mathbb{R}\) is a prescribed threshold for the normalized radial electric field.
Denote by \(D\subset\mathbb{R}^2\) the basin of attraction of the L-mode steady state \(N_3=(X_3,U_3)\), and let
\(\tau_{X_w}:=\inf\{t\ge0: X_t=X_w\}\) be the first exit time from \(D\) through \(\Sigma_w\).
Assume the weak-noise limit \(\varepsilon\ll1\) and let \(V(X,U)\) be the quasipotential obtained from the trained PINN.

The asymptotic behaviour of the mean first exit time (MFET)
\(\mathbb{E}_{N_3}[\tau_{X_w}]\) depends critically on whether \(\Sigma_w\) is characteristic or non-characteristic with respect to the deterministic drift field \(b(X,U)=(b_1,b_2)^\top\). If the drift field intersects \(\Sigma_w\) transversally, i.e.,
\[
\mu^*:=-\langle b(x^*),n\rangle
= -b_1(X_w,U^*)>0,
\qquad
x^*=(X_w,U^*),
\]
where \(U^*\) minimizes \(V(X_w,U)\) over \(U>0\),
then \(\Sigma_w\) is non-characteristic \cite{li2025rare} and
\begin{equation}
\mathbb{E}_{N_3}[\tau_{X_w}]
\sim
\frac{1}{\mu^*}
\sqrt{\frac{\det h^*}{2\pi\varepsilon}}
\exp\!\left\{\frac{V(X_w,U^*)-V(N_3)}{\varepsilon}\right\},
\qquad \varepsilon\to0.
\label{nonchar}
\end{equation}
Here \(h^*\) denotes the transverse Hessian of \(V\) at \(x^*\) along the \(U\)-direction,
and \(n=(1,0)^\top\) is the outward unit normal to \(\Sigma_w\).

If \(X_w=X_{N_2}\), where \(N_2=(X_{N_2},U_{N_2})\) is the saddle point, then \(\Sigma_w\) coincides with the stable manifold of \(N_2\) and becomes characteristic: \(\langle b(N_2),n\rangle=0\). In this case the exit occurs via the most probable path passing through the saddle, and
\begin{equation}
\mathbb{E}_{N_3}[\tau_{X_w}]
\sim
\frac{2\pi}{\sqrt{|\Lambda_L'|\,|\Lambda_m'|}\,
\Gamma_{\mathrm{coupling}}\,
\Gamma_{\mathrm{path}}}
\exp\!\left\{\frac{V(N_2)-V(N_3)}{\varepsilon}\right\},
\qquad \varepsilon\to0,
\label{cha}
\end{equation}
where
\(\Lambda_L'=\partial b_1/\partial X|_{N_3}\),
\(\Lambda_m'=\partial b_1/\partial X|_{N_2}\),
and \(\Gamma_{\mathrm{coupling}},\Gamma_{\mathrm{path}}\) are the coupling and geometric correction factors defined in Proposition~3.1.

Note that (\ref{nonchar}) follows from the Freidlin--Wentzell large deviation theory for non-characteristic boundaries \cite{Bouchet2016}, where the exit location concentrates at the quasipotential minimizer \(x^*\) on \(\Sigma_w\), and the prefactor arises from Gaussian fluctuations normal to the boundary. In contrast, (\ref{cha}) reduces to the Langer-Bouchet escape problem via a saddle point \cite{Langer1969},
which has already been established in Proposition~3.1 for the L-H transition system.

\section{Conclusions and future challenges}
We presented a two-dimensional coupled system for the L-H confinement transition, built upon the phenomenological coupling between zonal flows and turbulent fluctuations at the edge of tokamak plasmas. For this system, we derived the Hamilton-Jacobi equation and solved it numerically using a physics-informed neural network, incorporating a loss function designed via vector field decomposition. This approach yielded the system's quasipotential, which in turn enabled the computation of the most probable path for rare confinement-state transitions. Compared with classical two-dimensional L-H models, our system featured state variables that were more readily accessible to experimental measurements and rested on a more comprehensive physical foundation. Furthermore, it self-consistently captured the dynamical behavior of tokamak devices during the startup phase.

In future work, we will systematically extend the present two-dimensional L-H transition model toward greater physical rigor and experimental relevance. First, we will further validate the physical rationality of the parameter space employed in the deterministic and stochastic models, including the linear turbulence growth rate \(\alpha_1\), nonlinear shear suppression coefficient \(\beta\), coupling coefficient \(\zeta\), and neoclassical drive parameter \(X_{NC}\). By scanning a wider range of physically admissible parameters and comparing with plasma transport scalings and experimental constraints from tokamak operations, we will refine the parameter bounds to ensure consistency with neoclassical theory, turbulence regimes \cite{YBD}, and zonal flow dynamics, thus strengthening the foundational credibility of the bistable structure and transition landscape \cite{YB}. Second, we will conduct a comprehensive error analysis and error control study for the numerical results, especially the quasipotential computed by the physics-informed neural network and the most probable path obtained via the shooting method. Sources of error, including network approximation error, grid discretization error, automatic differentiation error, and residual tolerance in satisfying the Hamilton-Jacobi equation, will be quantified and decomposed. Adaptive training strategies, refined network architectures, and post-processing regularization will be adopted to improve numerical accuracy, stabilize convergence, and establish reliable error bounds for quantitative predictions of potential depths, barrier heights, and transition rates.

\bigskip
\noindent\textbf{Data availability}

The numerical algorithms and source code that support the findings of this study are available from the corresponding author upon reasonable request.

\bigskip
\noindent\textbf{Declaration of competing interest}

No author associated with this paper has disclosed any potential or pertinent conflicts which
may be perceived to have impending conflict with this work.

\bigskip
\noindent\textbf{Acknowledgments}

Warm thanks for the inspirations from Professor Zhihui Liu and Professor Liu Hong. This work was supported by the Guangdong Basic and Applied Basic Research Foundation (Grant No. 2025A1515012560), the Guangdong Introduction Program (Grant No. 2023QN10X753) and National
Foreign Experts Program (Grant No. 111001819820258003).

\end{document}